\documentclass[11pt,twoside]{article}
\usepackage{artmods,macros,amsmath,amssymb}
\usepackage{xcolor}
\usepackage{tikz,float,tkz-euclide,subcaption}
\usepackage{cite}
\input macros

\def\kdm1{_{k-1}\drop}
\let\span\Span

\title{An optimization derivation of the method of conjugate gradients}

\date{\today}

\author{David Ek\thanks{\footTO} \and Anders Forsgren\addtocounter{footnote}{-1}\footnotemark}

\def\footTO{Optimization and Systems Theory, Department of
  Mathematics, KTH Royal Institute of Technology, SE-100 44 Stockholm,
  Sweden ({\tt daviek@kth.se,andersf@kth.se}). \newline \indent \indent Research supported by the Swedish Research Council (VR).}

\begin{document}
\maketitle\thispagestyle{empty}

\begin{abstract}
We give a derivation of the method of conjugate gradients based on the
requirement that each iterate minimizes a strictly convex quadratic on the space spanned by the previously observed gradients. Rather than verifying that the search direction has the correct properties, we show that generation of such iterates is equivalent to generation of orthogonal gradients which gives the description of the direction and the step length.

Our approach gives a straightforward way to see that the search
direction of the method of conjugate gradients is a negative scalar
times the gradient of minimum Euclidean norm evaluated on the affine span of the
iterates generated so far.

  
\end{abstract}

\section{Introduction}
In this work we consider the classical method of conjugate gradients
as introduced by Hestenes and Steifel \cite{HS52}. For some history on
the development of methods of this type, see the review by Golub and
O'Leary \cite{GO89}. Let the function $q : \mathbb{R}^n \to
\mathbb{R}$ be defined by
\[
q(x) = \half x\T H x + c\T x,
\]
for $H \in \mathbb{R}^{n \times n}$ symmetric positive definite and $c \in \mathbb{R}^n$. The method of conjugate
gradients aims at finding a solution to $\grad q(x)=0$, i.e.,
$Hx+c=0$. Here and throughout,
$g_k$ denotes $\grad q(x_k)$, i.e., $g_k = Hx_k+c$. It is well known
that iteration $k$ of the method can be characterized by finding a
search direction $p_k^{CG}$, given by
\begin{subequations}\label{eqn-pkCG}
\begin{equation} \label{eqn-pkCGp} 
p^{CG}_k = - g_k + \frac{g_k^T g_k}{g_{k-1}^T g_{k-1}} p^{CG}_{k-1},
\text{with} p^{CG}_0=-g_0.
\end{equation}
The form (\ref{eqn-pkCGp}) can by expansion also be written in terms of gradients only as
\begin{equation}\label{eqn-pkCGg}
p_k^{CG}  =  - g_k^T g_k \sum_{i=0}^k \frac1{g_i^T g_i} g_i.
\end{equation}
\end{subequations}
The next iterate $x_{k+1}$ is found by exact linesearch. This means finding
$\theta_k$ such that $q(x_k+\theta p_k)$ is minimized, i.e., $p_k^T
(g_k + \theta_k H p_k)=0$. Consequently $\theta_k$ is given by
\begin{equation}\label{eqn-theta}
\theta_k=-\frac{g_k^T p_k}{p_k^T H p_k}.
\end{equation}
Finally $x_{k+1} = x_k + \theta_k p_k$. 

The derivation of $p_k^{CG}$ given by (\ref{eqn-pkCG}) is done in many
ways in the literature, see
e.g.,~\cite{NW06,GMW81,H80,P09,B15,PAENSO16,CGT00,F81,BC15,LY16,SY06,P04,AD74,CZ13,K95}
for derivations in an optimization framework and
e.g.,~\cite{GV13,S03,LS13,AK08,BKK17} for derivations in a numerical
linear algebra framework. In either case, the aim is typically to
directly derive $p_k^{CG}$ on the form (\ref{eqn-pkCGp}).

  
In an optimization framework, the approach is often for a given $x_0$ to generate a sequence of points $x_k$, $k=1,2,\dots$, where $x_k$ is characterized as the minimizer of $q(x)$ on an expanding subspace. In particular, a sequence of points $x_k$,
$k=1,2,\dots$, where $x_k$ is the minimizer of 
\begin{equation} \label{eq:OPk}
\begin{array}{cl} 
\text{minimize} { q(x)} \quad \text{subject to}  { x \in x_0 + \mathcal{S}_k,}
\end{array}
\end{equation}
where $x_0 + \mathcal{S}_k= \{ x : x = x_0 + s, \ s \in \mathcal{S}_k
\}$ for $\mathcal{S}_k$ as a $k$-dimensional space spanned by vectors
$s_0,\dots, s_{k-1}$, i.e., $\mathcal{S}_k = \span (\{s_0,
  \dots, s_{k-1}\} )$.  Typically $\mathcal{S}_k$ is
chosen as the space spanned by the generated search directions $p_0,
\dots, p_{k-1}$. The formula for $p_k^{CG}$ of (\ref{eqn-pkCGp}) is then
derived by utilizing relations which follow from $x_k$ being the
solution of (\ref{eq:OPk}), often under the assumption that the method
generates conjugate directions. There are also many derivations which
are based on verifying that $p_k$ of (\ref{eqn-pkCGp}) has the
required properties.

\section{A derivation of the method of conjugate gradients}
Our aim is to give a characterization of $x_k$ as found by generating
orthogonal gradients. Then the search direction $p_k^{CG}$ and the step length $\theta_k$ are derived
as a consequence of this characterization. This means that rather than
verifying that the formulas (\ref{eqn-pkCG}) and (\ref{eqn-theta}) are correct, we show
that they must take this form. 


Similarly to many other derivations, our approach is for a given $x_0$ to generate a sequence of points $x_k$, $k=1,2,\dots$, where $x_k$ is characterized as the minimizer of $q(x)$ on an expanding subspace. In contrast, we choose a sequence of points $x_k$, $k=1,2,\dots$, where $x_k$ is the minimizer of (\ref{eq:OPk}) with $\mathcal{S}_k = \span (\{ g_0,\dots, g_{k-1} \} )$, i.e., $\mathcal{S}_k$ is chosen as the space spanned by the observed gradients. With a basic lemma, we show that generating such $x_k$ is equivalent to orthogonality of gradients. In consequence, the information considered is in terms of iterates and gradients only.

In essence, the orthogonality of the gradients is viewed as the basic property of the method. This gives a way of deriving the the search direction as a consequence of the characterization.  To further emphasize the pedagogical nature of this derivation we also give a geometrical interpretation.

Our basic lemma gives a characterization of $x_k$ such that $g_k$ is
orthogonal to all previous gradients $g_i$, $i=0,\dots,k-1$. 
Note that there is a one-to-one
correspondence between $x_k$ and $g_k$, due to the nonsingularity of
$H$, by $g_k=Hx_k+c$ and $x_k = H\inv (g_k -c)$.  
Therefore, we may characterize $x_k$ by properties of $g_k$.
\begin{lemma}\label{lem-AF}
For a given initial point $x_0$, let $x_k$, $k=1,2,\dots$, be defined
as the minimizer of $q(x)$ subject to $x \in x_0 + \span ( \{g_0, \dots,
g_{k-1}\} )$, where $g_i=\grad q(x_i)$,  $i=0,1,\dots$. Then, $x_k$ is
well defined and may be characterized by
\[
g_k^T g_i=0, \quad i=0,\dots,k-1.
\]
In addition, there is an $r$, with $r\le n$, such that $g_r=0$ and
$g_k\ne 0$, $k=0,\dots,r-1$.
\end{lemma}
\begin{proof}
Any vector $x \in x_0 + \span ( \{g_0, \dots, g_{k-1} \} )$ can be written as $x(v) = x_0 + \sum_{j=0}^{k-1} v_j g_j$, for some $v_j \in \mathbb{R}$, $j=0,\dots,k-1$.  Since $H$ is symmetric and positive definite, $q(x)$ is a strictly convex function.  Therefore, for a given $k$, necessary and sufficient  conditions for $x_k = x(v^*)$ to minimize $q(x(v))$ are given by
  \[
\left. \frac{\partial q(x(v))}{\partial v_i} \right\vert_{v = v^*}  = 0, \quad i=0,\dots,k-1,
\]
which by the chain rule can be written as $\nabla q(x_0 + \sum_{j=0}^{k-1} v_j^* g_j)^Tg_i =0$, $i=0,\dots,k-1$. The result follows with $x_k = x_0 + \sum_{j=0}^{k-1} v^*_j g_j$. Finally, there can be at most $n$ nonzero orthogonal gradients
$g_i$. Therefore, $g_r=0$ for some $r$, $0\le r \le n$.
\end{proof}\\
As proclaimed, for a given $x_0$, the sequence $x_k$, $k=1,2,\dots$, in Lemma~\ref{lem-AF} is such that $x_k$, $k=1,2,\dots$, is characterized as the the solution of (\ref{eq:OPk}) with $\mathcal{S}_k =  \span ( \{g_0, \dots, g_{k-1}\} )$.

For a given initial point $x_0$, our interest is now iteration $k$, with $0\le k< r$, so that $g_i\ne
0$, $i=0,\dots,k$. We have $x_k$ with the properties of
Lemma~\ref{lem-AF} and want to generate $p_k$ and $\theta_k$ such that
$x_{k+1}=x_k+\theta_k p_k$ satisfies Lemma~\ref{lem-AF} with $k$
replaced by $k+1$. This means that we
want $\theta_k$ and $p_k$ such that $x_{k+1}$ gives $g_{k+1}^T g_i=0$,
$i=0,\dots,k$, with $x_{k+1} \in x_0 + \span ( \{g_0, \dots,
g_{k}\} )$. As $g_k \neq 0$ it follows from $g_{k+1}^T g_k = 0$
that $g_{k+1}\ne g_k$. Therefore, $x_{k+1}\ne x_k$ since
$x_{k+1}-x_k=H\inv (g_{k+1}-g_k)$, so that $\theta_k \neq 0$ and $p_k
\neq 0$. Moreover, since $x_{i+1}-x_0 \in \span (\{g_0, \dots,g_{i}\})$,
$i=0,\dots,k$, it must also hold that
%
\begin{equation}\label{eqn-CG}
g_{k+1}^T (x_{i+1}-x_0)=0, \quad i=0,\dots,k.
\end{equation}
We will show that the first $k$ equations of (\ref{eqn-CG}) give a characterization of the search direction and that the last equation gives the step length. As $g_{k+1} = \grad q (x_k+\theta_k p_k)= g_k + \theta_k Hp_k$, (\ref{eqn-CG}) takes
the form
\begin{equation}\label{eqn-CGII}
(g_k + \theta_k H p_k)^T(x_{i+1}-x_0)=0, \quad i=0,\dots,k.
\end{equation}
The first $k$ equations of (\ref{eqn-CGII}) are equivalent to
\begin{equation}\label{eqn-CGIII}
g_k^T(x_{i+1}-x_0)+\theta_k p_k^TH (x_{i+1}-x_0)=0, \quad i=0,\dots,k-1.
\end{equation}
Analogous to (\ref{eqn-CG}), we have $g_k^T(x_{i+1}-x_0)=0$,
$i=0,\dots,k-1$. As $\theta_k\neq 0$, (\ref{eqn-CGIII}) gives
\begin{equation}\label{eqn-CGIV}
p_k^TH (x_{i+1}-x_0)=0, \quad i=0,\dots,k-1,
\end{equation}
independently of $\theta_k$. Note that
\begin{equation}\label{eqn-CGV}
H(x_{i+1}-x_0)=Hx_{i+1}+c-(Hx_0+c)=g_{i+1}-g_0,
\end{equation}
so that a combination of (\ref{eqn-CGIV}) and (\ref{eqn-CGV}) gives
\begin{equation*}
p_k^TH (x_{i+1}-x_0)=p_k^T(g_{i+1}-g_0)=0, \quad i=0,\dots,k-1.
\end{equation*}
Consequently, it must hold that
\begin{equation}\label{eqn-cg1}
p_k^Tg_i=c_k, \quad i=0,\dots,k,
\end{equation}
for an arbitrary constant $c_k$. The choice of $c_k$ gives a scaling
of $p_k$. The condition given by (\ref{eqn-cg1}) is illustrated
geometrically for $k=1$ and $k=2$ in Figure~\ref{fig:fig}.

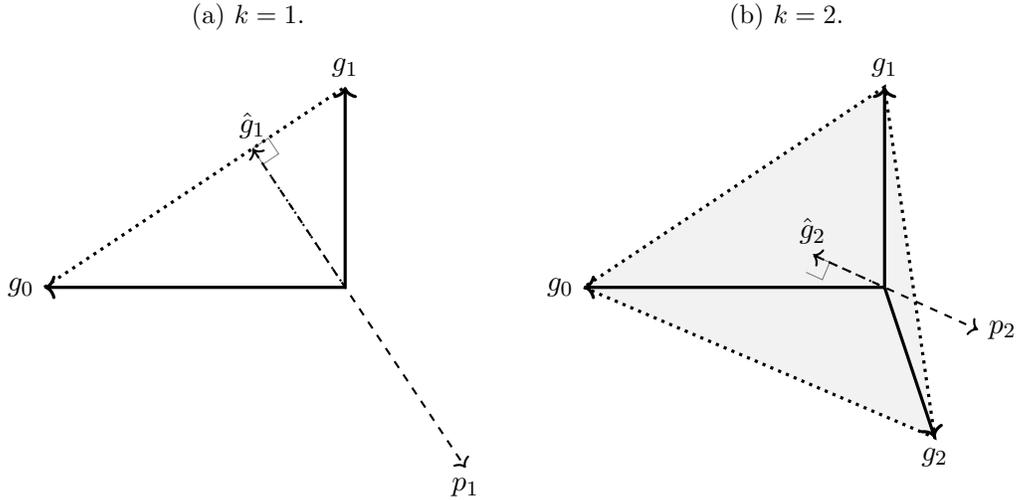
\begin{figure}[H]
    \centering
    \begin{subfigure}[t]{0.5\textwidth}
        \centering
          \caption{$k=1$.}
\begin{tikzpicture}
\coordinate[label=left:{$g_0$}]  (A) at (-12/3,0);
\coordinate (B) at (0,0);
\coordinate[label=above:{$g_1$}] (C) at (0,8/3);
\coordinate[label=below:{{\color{white} $g_2$}},white] (D) at (2/3,-6/3);

\draw[->,very thick] (B) -- (A);
\draw[->,very thick] (B) -- (C);
\draw[dotted,very thick] (A) -- (C);
\coordinate (a) at ($(B)!(A)!(C)$);
\coordinate (b) at ($(A)!(B)!(C)$);
\coordinate (c) at ($(A)!(C)!(B)$);
\draw[dotted, thick] (B) -- (b) node[above] {$\hat{g}_1$};
\tkzFindAngle(B,b,C)
\tkzGetAngle{angleBbC}; \FPround\angleBbC\angleBbC{0}
\tkzMarkRightAngle[draw=black,opacity=.5](B,b,C)
\draw [<->, add= 0 and 1.3, dashed, thick] (b) to (B) node[below] {$p_1$};
\end{tikzpicture}
        \label{fig:k2}
    \end{subfigure}%
    ~ 
    \begin{subfigure}[t]{0.5\textwidth}
     \caption{$k=2$.}
        \centering
\begin{tikzpicture}
\coordinate[label=left:{$g_0$}]  (A) at (-12/3,0);
\coordinate[label=above:{$g_1$}] (C) at (0,8/3);
\coordinate[label=below:{$g_2$}] (B) at (2/3,-6/3);
\filldraw[fill = gray!10, draw = white] (A) -- (B) -- (C) -- cycle;
\coordinate (D) at (0,0);
\coordinate (d) at ($(A)!(D)!(C)$);
\draw[->,very thick] (D) -- (B);
\draw[->,very thick] (D) -- (C);
\coordinate (a) at ($(B)!(A)!(C)$);
\coordinate (b) at ($(A)!(B)!(C)$);
\coordinate (c) at ($(A)!(C)!(B)$);
\coordinate (O) at (intersection of A--a and B--b);
\draw[dotted, thick] (D) -- (O) node[above] {$\hat{g}_2$};
\draw [<->, add= 0 and 1.3, dashed, thick] (O) to (D) node[right] {$p_2$};
\tkzMarkRightAngle[draw=black,opacity=.5](c,O,D)
\draw[gray!10,thick] (O) to (c); 
\draw[very thick,dotted] (A) -- (B) -- (C) -- cycle;
\draw[->,very thick] (D) -- (A);
\tkzFindAngle(B,b,C)
\end{tikzpicture}
        \label{fig:k2}
    \end{subfigure}
    \caption{Illustration of the search direction $p_k$ as the
      iterations proceed. The vector $\hat{g}_k$ represents the vector
      of minimum Euclidean norm in the affine span of
      $g_0$, $g_1$, \dots, $g_{k}$, or equivalently the gradient of minimum
      Euclidean norm evaluated on $x_0 + \span(\{g_0,
\dots, g_{k-1}\})$, as will be discussed later.}
    \label{fig:fig}
\end{figure}
It remains to characterize $p_k$ that satisfies (\ref{eqn-cg1}).  The
requirement that $x_{k+1}\in x_0 + \span (\{g_0, \dots, g_{k}\})$ for
$x_{k}\in x_0 + \span (\{g_0, \dots, g_{k-1}\})$ means that
\begin{equation}\label{eqn-p}
p_k= \sum_{j=0}^k \gamma_j g_j,
\end{equation}
where $\gamma_j$, $j=0,\dots,k$, are to be determined. A combination
of (\ref{eqn-cg1}) and (\ref{eqn-p}), using the orthogonality of the
$g_i$s, gives
\[
c_k = p_k^T g_i = \sum_{j=0}^k \gamma_j g_j^T g_i = \gamma_i g_i^T
g_i,
\]
so that $\gamma_i = \frac{c_k}{g_i^T g_i}$, $i=0,\dots,k,$ which gives
\begin{equation}\label{eqn-pk}
p_k = c_k \sum_{i=0}^k \frac1{g_i^T g_i} g_i.
\end{equation}
A comparison of $p_k$ given by (\ref{eqn-pk}) and $p_k^{CG}$ given by (\ref{eqn-pkCGg}) already shows that the directions are the same. The characterization of $p_k$ given by (\ref{eqn-pk}) is in terms of
all gradients $g_i$, $i=0,\dots,k$. To get the simple recursion of the method of conjugate gradients, we may use (\ref{eqn-pk}) to write
\begin{eqnarray*}
p_k & = & c_k \sum_{i=0}^k \frac1{g_i^T g_i} g_i
= \frac{c_k }{g_k^T g_k} g_k + c_k \sum_{i=0}^{k-1} \frac1{g_i^T
  g_i} g_i \\
& = &
\frac{c_k }{g_k^T g_k} g_k + \frac{c_k}{c_{k-1}} \left( c_{k-1}\sum_{i=0}^{k-1} \frac1{g_i^T
  g_i} g_i\right) \\
& = &
\frac{c_k }{g_k^T g_k} g_k + \frac{c_k}{c_{k-1}} p_{k-1},
\text{with}
p_0=\frac{c_0}{g_0^T g_0} g_0.
\end{eqnarray*}
The choice $c_k=-g_k^T g_k$ gives the scaling of the method of
conjugate gradients, i.e., $p_k^{CG}$ of (\ref{eqn-pkCG}),
which is the standard way of describing the search direction of the
method of conjugate gradients. Finally, it remains to determine $\theta_k$. The last equation of (\ref{eqn-CGII}) in combination with $x_{k+1}=x_k+\theta_k p_k$ gives
\begin{equation} \label{eqn-thetak}
(g_k + \theta_k H p_k)^T(x_{k}-x_0) + \theta_k (g_k + \theta_k H p_k)^Tp_k =0.
\end{equation}
By the choice of $p_k$ it holds that $(g_k + \theta_k H p_k)^T(x_{k}-x_0)=0$. Moreover, since $\theta_k \neq 0$, by (\ref{eqn-thetak}) it must hold that $(g_k + \theta_k H p_k)^Tp_k =0$ which gives (\ref{eqn-theta}). 

Therefore, by our approach we have derived two equivalent ways of writing the
search direction of the method of conjugate gradients. In addition, we have shown that exact linesearch also follows from the characterization given in Lemma~\ref{lem-AF}.

\section{Relationship to the minimum Euclidean norm gradient}

A benefit of deriving the formula for $p_k$ given by (\ref{eqn-pkCGg})
is that it gives a straightforward way to relate $p_k^{CG}$ to the
vector of minimum Euclidean norm in the affine span of $g_0$, $g_1$,
\dots, $g_k$, which we will show is equivalent to the gradient of
minimum Euclidean norm evaluated on the affine span of $x_0$, $x_1$, \dots,
$x_k$. As in the previous section, we will consider iteration $k$,
with $k<r$. For $k=r$, $g_r=0$ immediately gives this minimum norm
vector zero. As we will show, this minimum-norm vector $\ghat_k$ is
nonzero and may be characterized by
\begin{equation}\label{eqn-ghat1}
\ghat_k^T(g_i - \ghat_k)=0, \quad i=0,\dots,k,
\end{equation}
which is what is illustrated in Figure~\ref{fig:fig} for $k=1$ and
$k=2$. By (\ref{eqn-ghat1}), $\ghat_k^T g_i = \ghat_k^T \ghat_k$,
$i=0,\dots,k$, so that $\ghat_k$ satisfies (\ref{eqn-cg1}) with
$c_k=\ghat_k^T \ghat_k$.  Consequently,
\begin{equation}\label{eqn-cgmr}
p_k^{CG}=-\frac{g_k^T g_k}{\ghat_k^T \ghat_k} \ghat_k.
\end{equation}
It remains to verify that a vector $\ghat_k$, which belongs to the
affine span of $g_0$, $g_1$, \dots, $g_k$ and satisfies
(\ref{eqn-ghat1}), is a nonzero vector of minimum Euclidean norm in the affine
span of $g_0$, $g_1$, \dots, $g_k$. Note that if $\sum_{i=0}^k
\alpha_i =1$, then
\begin{equation}\label{eqn-g}
g = \sum_{i=0}^k \alpha_i g_i = \sum_{i=0}^k \alpha_i (H x_i + c)
= H (\sum_{i=0}^k \alpha_i x_i) + (\sum_{i=0}^k
  \alpha_i) c = Hx+c,
\end{equation}
for
\begin{equation}\label{eqn-x}
x= \sum_{i=0}^k \alpha_i x_i = x_0 + \sum_{i=1}^k \alpha_i (x_i-x_0).
\end{equation}
A combination of (\ref{eqn-g}) and (\ref{eqn-x}) shows that any $g$
which is given by an affine combination of $g_0$, $g_1$, \dots, $g_k$
is the gradient of $q(x)$ at an $x$ given by the corresponding affine
combination of $x_0$, $x_1$, \dots, $x_k$. In addition, by
(\ref{eqn-x}), the affine span of $x_0$, $x_1$, \dots, $x_k$ is
equivalent to $x_0 + \span (\{x_1-x_0, \dots, x_k-x_0\})$. Finally, by
the construction of $x_i$, $i=1,\dots,k$, it follows that $x_0 +
\span (\{x_1-x_0, \dots, x_k-x_0\})$ is equivalent to $x_0 + \span(\{g_0,
\dots, g_{k-1}\})$. Consequently, since $k<r$, it must hold that
$\ghat_k\ne 0$. In particular, a combination of (\ref{eqn-pk}) and
(\ref{eqn-cgmr}) shows that for $\ghat_k$ to belong to the affine span
of $g_0$, $g_1$, \dots, $g_k$, we must have
\begin{equation}\label{eqn-ghat2}
\ghat_k = \frac1{\sum_{j=0}^k \frac1{g_j^T g_j}}\sum_{i=0}^k \frac1{g_i^T g_i}
g_i.
\end{equation}
It is
straightforward to verify that $\ghat_k^T g_i = \ghat_k^T \ghat_k$,
$i=0,\dots,k$, holds with $\ghat_k$ given by (\ref{eqn-ghat2}), so the
description of $\ghat_k$ given by (\ref{eqn-ghat2}) is consistent with
(\ref{eqn-ghat1}).
In addition, for $g=\sum_{i=0}^k \alpha_i
g_i$ with $\sum_{i=0}^k \alpha_i=1$, we have
\begin{equation}\label{eqn-orth}
(g-\ghat_k)^T \ghat_k = (\sum_{i=0}^k \alpha_i g_i -\ghat_k)^T\ghat_k
 = \sum_{i=0}^k \alpha_i (g_i -\ghat_k)^T\ghat_k=0,
\end{equation}
where (\ref{eqn-ghat1}) has been used in the last step.
Therefore, by (\ref{eqn-orth}) it follows that
\begin{eqnarray*}
g\T g & = & (\ghat_k + (g-\ghat_k))^T (\ghat_k + (g-\ghat_k))
\nonumber \\ \label{eqn-normg}
& = &
\ghat_k^T \ghat_k + 2 (g-\ghat_k)^T \ghat_k 
+ (g-\ghat_k)^T (g-\ghat_k) \\
& = &\ghat_k^T \ghat_k  + (g-\ghat_k)^T (g-\ghat_k) \ge \ghat_k^T \ghat_k,
\end{eqnarray*}
so that $\ghat_k$ is indeed the vector of minimum Euclidean norm in
the affine span of $g_0$, $g_1$, \dots, $g_k$, or equivalently, the
gradient of minimum Euclidean norm evaluated on 
$x_0 + \span(\{g_0, \dots, g_{k-1}\})$.

Our description of $\ghat_k$ is based on \cite{FO14}. Other
derivations of the result on the connection between $p_k^{CG}$ and
$\ghat_k$, which is the search direction in the method of shortest
residuals, can be found in
e.g.,~\cite{H80,P09}.

\section{Summary}
Our aim has been to derive the method of conjugate gradients in a
straightforward manner based on gradients only. The approach
facilitates a geometrical interpretation of the consecutive search
directions. We have limited the discussion to the derivation and the
relationship to the gradient of minimum Euclidean norm evaluated on
the affine span of the iterates generated so far. Conventional results on
conjugacy and relations to Krylov subspaces can henceforth be derived
in a straightforward fashion.

\bibliography{refs}
\bibliographystyle{myplain}

\end{document}